\documentclass[12pt,oneside]{amsart}

\usepackage[section]{placeins} 
\usepackage[margin=1in]{geometry}
\usepackage{mathabx}
\usepackage{xcolor}
\usepackage{graphicx}
\usepackage{amssymb}
\usepackage{tikz}
\usepackage{latexsym}
\usepackage{epsfig}
\usepackage{enumitem}
\usepackage{multicol}
\usepackage{wasysym}
\usepackage{amsmath}
\usepackage{amsthm}
\usepackage{amscd}
\usepackage{ulem}
\usepackage{soul}
\usepackage{multirow}
\usepackage{rotating}
\usepackage{pdflscape}
\usepackage{float}
\usepackage{dsfont}
\usepackage{comment}

\usepackage[colorlinks,citecolor=red,pagebackref,hypertexnames=false]{hyperref}
\hypersetup{
  colorlinks = true,
  citecolor=red,
  linkcolor  = red
}

\setlist{itemsep=.06125in}

\restylefloat{table}
\numberwithin{equation}{section}

\theoremstyle{plain}
\newtheorem{theorem}{Theorem}[section]
\newtheorem{lemma}[theorem]{Lemma}

\newtheorem{proposition}[theorem]{Proposition}

\theoremstyle{definition}

\newtheorem{example}[theorem]{Example}

\theoremstyle{remark}
\newtheorem{remark}[theorem]{Remark}

\newcommand{\eps}{\varepsilon}

\date{\today}

\author{A. Iosevich, E. Palsson, and A. Yavicoli}
\address{Department of Mathematics, University of Rochester, Rochester, NY, USA}
\email{iosevich@gmail.com}
\address{Department of Mathematics, Virginia Tech, Blacksburg, VA, USA}
\email{palsson@vt.edu}
\address{Department of Mathematics, University of British Columbia, Vancouver, BC, Canada}
\email{yavicoli@math.ubc.ca}

\thanks{A. I. was supported in part by the National Science Foundation under NSF DMS - 2154232.}
\thanks{A. Y. was supported in part by the Natural Sciences and Engineering Research Council of Canada, NSERC (GR030571 and GR030540).}

\title{Discretization, sampling, and the Fourier ratio}

\begin{document}

\maketitle

\begin{abstract}
We derive fundamental sampling bounds for smooth signals in continuous settings without sparsity assumptions. By introducing the Fourier ratio as a measure of spectral compressibility induced by smoothness, we obtain explicit, deterministic bounds linking signal regularity to recoverability from incomplete random samples. For functions in $C^{2}([0,1]^{2})$ sampled on an $N$ by $N$ grid, we show that a random subset of spatial samples of size
$$
C\frac{r_{N}^{2}}{\eps^{2}}\log(r_{N}/\eps)^{2}\log(N^{2})
$$
suffices, with high probability, to recover the entire discretized signal via $\ell^{1}$ minimization with relative $L^{2}$ error $O(\eps)$. We develop a parallel theory for bandlimited functions on the unit sphere, obtaining analogous recovery guarantees with sample complexity scaling polylogarithmically in the bandwidth. Our results establish smoothness as a deterministic prior that enforces compressibility in the Fourier domain, bridging continuous harmonic analysis with discrete compressed sensing in a unified information-theoretic framework.
\end{abstract}

\tableofcontents

\section{Introduction}

The purpose of this paper is to show that sampling problems in a variety of continuous settings can be effectively studied using discretization, followed by a reduction to discrete compressed sensing techniques where the Fourier ratio, developed in \cite{A2025}, is the key controlling parameter. This viewpoint fits into a broader compressed sensing framework based on stable recovery from incomplete measurements via $\ell^1$ minimization and related convex programs, as initiated in \cite{CRT06} and developed further in, for example, \cite{FR13,Rau10,RV08}.

A key point throughout the paper is that no sparsity assumption is imposed on the underlying signal in physical space. Instead, a quantitative form of compressibility emerges deterministically from regularity through bounds on the Fourier ratio, in the spirit of uncertainty principle and structural viewpoints in signal recovery \cite{DS89,IM24}.

The Fourier ratio, defined as $FR(f)=\|\widehat{f}\|_1/\|\widehat{f}\|_2$, serves as a measure of numerical sparsity or compressibility in the Fourier domain. A small Fourier ratio indicates that the Fourier coefficients are concentrated on a relatively small set, enabling stable recovery from incomplete samples. Crucially, we show that smoothness of the underlying continuous function directly controls this ratio, providing a deterministic bridge between regularity and compressibility.

We shall consider two scenarios, one on the unit cube in ${\mathbb R}^2$, and the other on the unit sphere in ${\mathbb R}^3$. As the reader shall see, both are easily generalizable to a variety of geometric settings.

Throughout the paper, statements of the form ``with high probability'' refer to the underlying sampling randomness and mean that the stated event holds with probability at least $1-D^{-c}$ for some absolute constant $c>0$, where $D$ denotes the relevant ambient dimension, namely $D=N^d$ in the discrete torus setting and $D=(L+1)^2$ in the spherical harmonic setting. The precise value of $c$ is not important for our purposes and can be tracked through the standard random sampling arguments in the references.

From an information-theoretic perspective, the Fourier ratio serves as a quantitative measure of the effective dimensionality of a signal in the Fourier domain. Smoothness enforces spectral decay, which in turn bounds the Fourier ratio and thereby limits the number of measurements required for stable recovery. This provides a deterministic alternative to sparsity-based compressibility, linking classical harmonic analysis with modern sampling theory. In this work, we show that smoothness alone, quantified by membership in $C^{2}$, implies such a bound, yielding explicit sample complexity guarantees for recovery from random samples without imposing any sparsity assumptions.

Much of the classical literature on sampling and recovery is formulated in terms of sparsity or approximate sparsity of coefficient sequences. In the present paper, however, we organize the analysis around the Fourier ratio as a quantitative measure of effective dimension. In companion work \cite{BIN2026}, we show through explicit examples, obstruction results, and complexity bounds that the Fourier ratio can often serve as a more accurate proxy for intrinsic signal complexity than approximate sparsity in time--frequency and imputation problems. The goal here is more limited and complementary: rather than comparing sparsity--based and Fourier--ratio--based models, we focus on explaining how bounds on the Fourier ratio arise naturally from discretization, smoothness, and geometric regularity assumptions, independent of sparsity considerations.

\subsection{A compressed sensing narrative driven by smoothness}

The guiding message of this paper can be summarized as follows. In many sampling problems the unknown object is not sparse in physical space, but rather arises from a smooth continuous model. After discretization, smoothness forces the discretized data to be compressible in a spectral basis, and the degree of compressibility is quantified by the Fourier ratio. Once this deterministic compressibility estimate is available, one can invoke standard $\ell^1$ recovery theorems to obtain stable reconstruction from incomplete measurements with explicit error bounds.

In the Euclidean setting, the unknown is a discretized field $g$ on ${\mathbb Z}_N^2$ obtained by sampling a function $f\in C^2([0,1]^2)$ on an $N$ by $N$ grid. Theorem \ref{thm:missing-values-unit-square} states that a random subset of point evaluations of $g$ of cardinality on the order of
$$
\frac{r_N^2}{\eps^2}\log(r_N/\eps)^2\log(N^2)
$$
suffices, with high probability, to recover the entire discretized field with relative $L^2$ error $O(\eps)$ by a standard basis pursuit type program. The key point is that $r_N$ is controlled deterministically in terms of $\|f\|_{C^2([0,1]^2)}$, $\|f\|_{L^2([0,1]^2)}$, and $N$, so no sparsity hypothesis is imposed.

\begin{remark}[Normalization of empirical norms]
In the recovery theorems below, the empirical norms on sampled sets are defined without a $|X|^{-1/2}$ normalization. This matches the normalization used in the standard restricted isometry and stability arguments for the corresponding sampling matrices. If one prefers averaged empirical norms, this amounts to a harmless rescaling of the constraint parameter $\eps$ and does not affect the sample complexity bounds except through absolute constants.
\end{remark}

In the spherical setting, the unknown is a bandlimited function $f\in V_L$ on $S^2$, and point evaluations model incomplete measurements distributed on the sphere. Theorem \ref{thm:sphere_missing_values} gives a parallel statement: if $f$ also belongs to $C^2(S^2)$, then an explicit spherical Fourier ratio bound $FR_L(f)\le r_L$ follows from regularity and bandwidth, and a random sample of size on the order of
$$
\frac{r_L^2}{\eps^2}\log(r_L/\eps)^2\log((L+1)^2)
$$
suffices for stable recovery in $L^2(S^2)$ by the same $\ell^1$ minimization paradigm. In both settings, smoothness plays the role that sparsity plays in more classical compressed sensing models, by enforcing spectral concentration in a quantitative way.

From the perspective of imaging sciences and inverse problems, the results of this paper provide a deterministic mechanism by which regularity of an underlying signal enforces quantitative compressibility in a transform domain, leading to stable reconstruction from incomplete spatial measurements. In contrast to classical compressed sensing frameworks, where sparsity is imposed as an explicit modeling assumption, no sparsity hypothesis is assumed here. Instead, sparsity emerges implicitly and in a controlled way from smoothness through explicit bounds on the Fourier ratio.

This viewpoint applies naturally to gridded spatial data arising in imaging and signal processing, as well as to data defined on curved geometries such as the sphere, which appear in applications including geophysical imaging, remote sensing, and directional data analysis. In these settings, the problem of recovering missing or corrupted measurements can be viewed as an inverse problem with incomplete sampling, and the results of this paper provide explicit recovery guarantees via standard $\ell^1$ minimization procedures. The emphasis is not on proposing new reconstruction algorithms, but rather on establishing mathematically rigorous conditions under which existing convex optimization methods provably succeed in recovering the underlying signal from partial data.

\subsection{Relation to frame-based erasure-resilient design}

The problem of recovering a field from incomplete sensor measurements is also studied in the context of frame theory and erasure-resilient sensor network design \cite{CasazzaKutyniok2013,BodmannPaulsen2005}. In that setting, sensors are modeled as elements of a finite frame, and conditions are derived, for example coherence bounds and restricted isometry properties, under which the signal can be recovered despite sensor failures. Our approach differs in several key aspects. While frame-based methods often assume the signal lies in a finite-dimensional space or is sparse in a known basis, we start from a continuous smooth signal and show that its discretization automatically satisfies a Fourier ratio bound, implying compressibility without an a priori sparsity assumption. Moreover, while frame-theoretic designs often focus on deterministic sensor placement to guarantee recovery for all signals in a given space, our results demonstrate that for smooth signals, random sensor placement suffices with high probability, and the required number of sensors scales polylogarithmically with the resolution. The Fourier ratio thus serves as a continuous analog of certain frame bounds, linking harmonic analysis to finite-frame erasure resilience while maintaining a direct connection to the underlying continuous model.

\subsection{Data distribution on the unit square}

Let $f\in C^2({[0,1]}^2)$. Consider an evenly spaced $N$ by $N$ grid on ${[0,1]}^2$, restrict $f$ to the vertices of this grid, and view the resulting function $g$ as a function from ${\mathbb Z}_N^2$ to ${\mathbb C}$. We shall view the points of ${\mathbb Z}_N^2$, embedded in the way we described, as our sensors, and the values of $g$ on ${\mathbb Z}_N^2$ as the sampling values on those sensors.

Our first result is the following.

\begin{theorem} \label{thm:missing-values-unit-square}
Let $f$ be a real-valued function on $[0,1]^2$ which is 1-periodic in each variable and belongs to $C^2([0,1]^2)$. Fix an integer $N\ge 2$ and define a function $g:{\mathbb Z}_N^2\to{\mathbb R}$ by
$$
g(x_1,x_2)=f(x_1/N,x_2/N).
$$
Assume that
$$
N\|f\|_{L^2([0,1]^2)}^2\ge 8\|f\|_{C^2([0,1]^2)}^2.
$$
Define the discrete Fourier transform of $g$ by
$$
\widehat g(m)=\frac{1}{N}\sum_{x\in{\mathbb Z}_N^2}\chi(-x\cdot m)g(x),
$$
where here and throughout, $\chi(t)=e^{2\pi i t/N}$. Note that this matches the normalization in Theorem \ref{thm:ZNd_recovery} when $d=2$, since $N^{-d/2}=N^{-1}$.

Define the Fourier ratio by
$$
FR(g)=\frac{{\|\widehat g\|}_1}{{\|\widehat g\|}_2}.
$$

Let
$$
A_N=2\frac{\left|\int_{[0,1]^2} f(x)dx\right|}{\|f\|_{L^2([0,1]^2)}},
\quad
B_N=16\pi^2\frac{\|f\|_{C^2([0,1]^2)}}{\|f\|_{L^2([0,1]^2)}}\log N,
\quad
C_N=\frac{8\pi^2}{\|f\|_{L^2([0,1]^2)}}\frac{\|f\|_{C^2([0,1]^2)}}{N}.
$$

Define
$$
r_N=A_N+B_N+C_N.
$$
Fix $\eps\in(0,1/2)$. By Proposition \ref{prop:discretization}, the discretized function $g$ satisfies
$$
FR(g)\le r_N.
$$

Let $X\subset{\mathbb Z}_N^2$ be chosen uniformly at random among all subsets of cardinality
$$
|X|=C\frac{r_N^2}{\eps^2}\log(r_N/\eps)^2\log D,
$$
where $D=N^2$, $C$ is a sufficiently large absolute constant. Assume that the values $g(x)$ are observed for $x\in X$ and missing for $x\in{\mathbb Z}_N^2\setminus X$. Define the empirical norm by
$$
\|h\|_{L^2(X)}=\left(\sum_{x\in X}|h(x)|^2\right)^{1/2}.
$$
Let $g^*:{\mathbb Z}_N^2\to{\mathbb R}$ be a solution to the convex optimization problem
$$
\min_{h:{\mathbb Z}_N^2\to{\mathbb R}}{\|\widehat h\|}_1
$$
where the minimization is taken subject to the constraint
$$
\|g-h\|_{L^2(X)}\le\eps\|g\|_{L^2({\mathbb Z}_N^2)}.
$$
If $C$ is sufficiently large, then with high probability
$$
\|g^*-g\|_{L^2({\mathbb Z}_N^2)}\le 11.47\eps\|g\|_{L^2({\mathbb Z}_N^2)}.
$$
In particular, the missing values of $g$ on ${\mathbb Z}_N^2\setminus X$ can be recovered with controlled relative $L^2$ error by applying Theorem \ref{thm:ZNd_recovery} to the discretized data.
\end{theorem}

\begin{remark}[Sampling and imputation viewpoint]
Theorem \ref{thm:missing-values-unit-square} can be read as a missing data recovery statement for gridded measurements. One starts with a smooth continuous field $f$ and observes its discretization $g$ only on a randomly chosen subset $X$ of the grid points, modeling random sensor failures or missing pixels. The conclusion is that the full grid of values can be reconstructed stably from the incomplete measurements by solving a standard $\ell^1$ minimization problem in the Fourier domain. The role of Proposition \ref{prop:discretization} is to provide an explicit deterministic bound $FR(g)\le r_N$ from smoothness, so that compressibility is guaranteed by regularity of the underlying continuous model rather than imposed as an a priori sparsity assumption.
\end{remark}

\begin{figure}[h]
\centering
\begin{tikzpicture}[
  node distance=2.15cm,
  box/.style={draw, rounded corners, align=center, minimum width=8.4cm, text width=8.0cm, inner sep=7pt}
]
\node (a) [box]
{Continuous model (imaging viewpoint) \\ underlying field $f$ on $[0,1]^2$ \\ $f\in C^2([0,1]^2)$};
\node (b) [box, below of=a]
{Measurement model \\ samples on a sensor grid \\ $g(x_1,x_2)=f(x_1/N,x_2/N)$ on ${\mathbb Z}_N^2$};
\node (c) [box, below of=b]
{Deterministic structure \\ regularity implies bounded Fourier ratio \\ $FR(g)\le r_N$ (Proposition \ref{prop:discretization})};
\node (d) [box, below of=c]
{Incomplete data \\ observed set $X\subset{\mathbb Z}_N^2$ chosen at random \\ $|X|=C\frac{r_N^2}{\eps^2}\log(r_N/\eps)^2\log(N^2)$};
\node (e) [box, below of=d]
{Reconstruction (inverse problem) \\ recover $g$ from $\{g(x):x\in X\}$ by convex optimization \\ minimize $\|\widehat h\|_1$ subject to data fidelity};
\node (f) [box, below of=e]
{Stability guarantee \\ $\|g^*-g\|_{L^2({\mathbb Z}_N^2)}\le 11.47\eps\|g\|_{L^2({\mathbb Z}_N^2)}$ \\ (Theorem \ref{thm:ZNd_recovery})};
\draw[->] (a) -- (b);
\draw[->] (b) -- (c);
\draw[->] (c) -- (d);
\draw[->] (d) -- (e);
\draw[->] (e) -- (f);
\end{tikzpicture}
\caption{Pipeline in inverse-problem form. A smooth continuous field is sampled on a grid, a Fourier ratio bound is obtained deterministically from regularity, and stable recovery from incomplete samples follows from $\ell^1$ minimization with an explicit error bound.}
\end{figure}
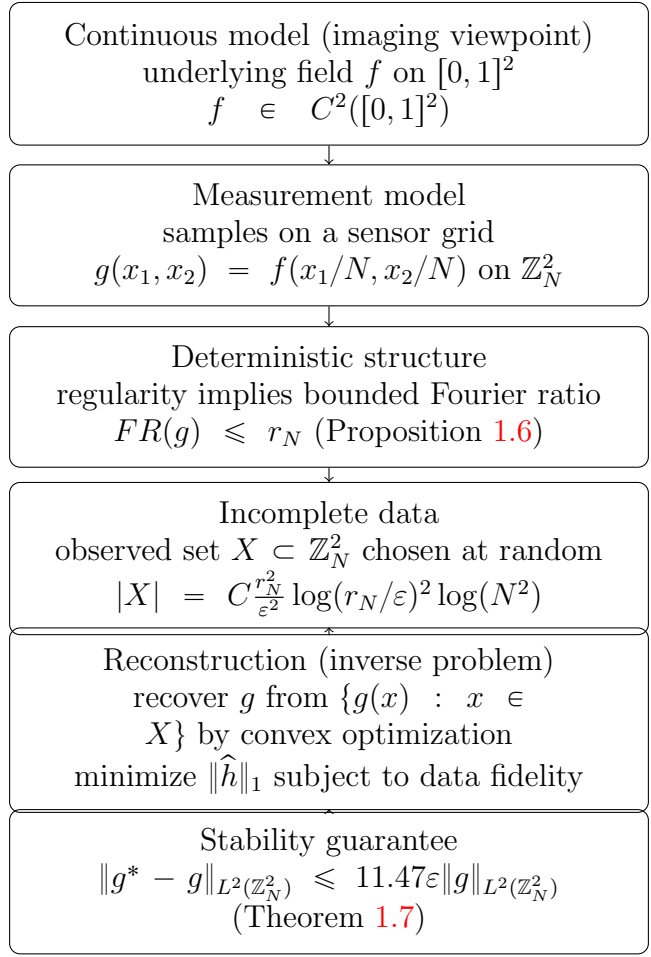

\begin{remark}[Algorithmic interpretation]
The convex optimization problems appearing in Theorems \ref{thm:missing-values-unit-square}, \ref{thm:ZNd_recovery}, and \ref{thm:sphere_missing_values} are standard $\ell^1$ minimization problems of the type commonly used in compressed sensing and inverse problems, often referred to as basis pursuit with noise. No new reconstruction algorithm is proposed here. Rather, the contribution of this paper is to provide deterministic and explicit conditions, expressed in terms of the Fourier ratio, under which such existing convex optimization procedures provably succeed in recovering a signal from incomplete measurements.
\end{remark}

\begin{remark}[Interpretation of the Fourier ratio bound]
Theorem \ref{thm:missing-values-unit-square} illustrates how the Fourier ratio framework developed in this paper leads to concrete recovery guarantees for discretized data with missing values. The result starts from a function on the unit square with mild smoothness, modeled by membership in $C^2([0,1]^2)$, and considers its discretization on a uniform grid.

Proposition \ref{prop:discretization} shows that the discretized function $g$ inherits an explicit bound on the Fourier ratio from the regularity of the underlying function $f$, provided the displayed lower bound assumption holds. In particular, it provides a deterministic estimate of the form
$$
FR(g)\le r_N.
$$
This estimate depends only on the smoothness of $f$ and the grid resolution and does not rely on any sparsity assumption. The three terms in $r_N$ have clear interpretations. The term $A_N$ captures the contribution of the mean value, the term $B_N$ reflects the logarithmic growth due to smoothness across frequencies, and the term $C_N$ represents a discretization error that decays with $N$.

Theorem \ref{thm:ZNd_recovery} then implies that any function on ${\mathbb Z}_N^2$ with a bounded Fourier ratio can be recovered stably from a random subset of its values by minimizing the $\ell^1$ norm of its Fourier transform. Theorem \ref{thm:missing-values-unit-square} combines these two results to yield a missing values recovery theorem for discretized smooth data. It guarantees that, with high probability, the entire discretized signal can be reconstructed from a number of observed samples that grows only polylogarithmically with the grid size, with an explicit and stable error bound.

From an applied perspective, this provides a rigorous justification for imputing missing measurements in gridded data sets using compressed sensing techniques \cite{CRT06,FR13}, without assuming sparsity a priori. Instead, sparsity emerges implicitly from regularity through the Fourier ratio.
\end{remark}

\begin{proposition} \label{prop:discretization}
Let $f$ be a real-valued function on ${[0,1]}^2$ which is 1-periodic in each variable and belongs to $C^2({[0,1]}^2)$. For a positive integer $N$, define a function $g:{\mathbb Z}_N^2\to{\mathbb R}$ by
$$
g(x_1,x_2)=f(x_1/N,x_2/N), \qquad (x_1,x_2)\in{\mathbb Z}_N^2.
$$
Assume that
$$
N\|f\|_{L^2([0,1]^2)}^2\ge 8\|f\|_{C^2([0,1]^2)}^2.
$$
Define the discrete Fourier transform of $g$ by
$$
\widehat g(m)=\frac{1}{N}\sum_{x\in{\mathbb Z}_N^2}\chi(-x\cdot m)g(x).
$$
Then
$$
FR(g)=\frac{{\|\widehat g\|}_1}{{\|\widehat g\|}_2}
$$
satisfies the bound
$$
FR(g)\le
2\frac{\left|\int_{[0,1]^2} f(x)dx\right|}{\|f\|_{L^2([0,1]^2)}}
+16\pi^2\frac{\|f\|_{C^2([0,1]^2)}}{\|f\|_{L^2([0,1]^2)}}\log N
+\frac{8\pi^2}{\|f\|_{L^2([0,1]^2)}}\frac{\|f\|_{C^2([0,1]^2)}}{N}.
$$
In particular, if $f\ge 0$ and $\mu=\int_{[0,1]^2} f(x)dx>0$, then
$$
FR(g)\le
2+16\pi^2\frac{\|f\|_{C^2([0,1]^2)}}{\mu}\log N
+\frac{8\pi^2}{\mu}\frac{\|f\|_{C^2([0,1]^2)}}{N}.
$$
\end{proposition}

\begin{theorem}[Recovery on ${\mathbb Z}_N^d$] \label{thm:ZNd_recovery}
Let $N$ and $d$ be positive integers and let $f:{\mathbb Z}_N^d\to{\mathbb R}$ be a function. Define the discrete Fourier transform of $f$ by
$$
\widehat f(m)=\frac{1}{N^{d/2}}\sum_{x\in{\mathbb Z}_N^d}\chi(-x\cdot m)f(x).
$$

Throughout the paper, all discrete Fourier transforms are normalized by the factor $N^{-d/2}$ so that Parseval’s identity holds in the form $\|\widehat{f}\|_2=\|f\|_{L^2(\mathbb{Z}_N^d)}$, and specializations to lower dimensions (such as $d=2$) are understood in this convention.

With this normalization Parseval's identity takes the form
$$
\|\widehat f\|_2=\|f\|_{L^2({\mathbb Z}_N^d)}.
$$
Define the Fourier ratio by
$$
FR(f)=\frac{\|\widehat f\|_1}{\|\widehat f\|_2}.
$$
Fix parameters $\eps\in(0,1/2)$ and $r\ge 1$, and assume that $FR(f)\le r$. Let
$$
D=N^d.
$$
Let $X\subset{\mathbb Z}_N^d$ be a random subset of cardinality
$$
|X|=C\frac{r^2}{\eps^2}\log(r/\eps)^2\log D,
$$
chosen uniformly at random among all subsets of this cardinality. Define the empirical norm
$$
\|g\|_{L^2(X)}=\left(\sum_{x\in X}|g(x)|^2\right)^{1/2}.
$$
Let $f^*:{\mathbb Z}_N^d\to{\mathbb R}$ be a solution to the convex optimization problem
$$
\min_{h:{\mathbb Z}_N^d\to{\mathbb R}}\|\widehat h\|_1
\quad\text{subject to}\quad
\|f-h\|_{L^2(X)}\le \eps\|f\|_{L^2({\mathbb Z}_N^d)}.
$$
If $C$ is a sufficiently large absolute constant, then with high probability, one has
$$
\|f^*-f\|_{L^2({\mathbb Z}_N^d)}\le 11.47\, \eps\|f\|_{L^2({\mathbb Z}_N^d)}.
$$
\end{theorem}

\begin{remark}[Relation to sparse spherical harmonic compressed sensing]
Rauhut and Ward \cite{RW12} and related work on spherical harmonic compressed sensing show that if the spherical harmonic coefficient vector $\widehat f$ is $s$-sparse or suitably compressible, then $f$ can be recovered stably from $q$ random point evaluations on $S^2$ via $\ell^1$ minimization, with $q$ scaling proportionally to $s$ up to polylogarithmic factors, after the standard bounded-orthonormal-system or preconditioning setup for spherical harmonics.

Theorem \ref{thm:sphere_missing_values} is compatible with that framework, but replaces an imposed sparsity hypothesis by an a priori Fourier ratio bound. Indeed, if $\widehat f$ is supported on at most $s$ indices, then $\|\widehat f\|_1\le \sqrt{s}\|\widehat f\|_2$, hence $FR_L(f)\le \sqrt{s}$, so Theorem \ref{thm:sphere_missing_values} yields the same $q\sim s$ scaling up to logarithmic factors in the sparse case. The novelty here is that Proposition \ref{prop:sphere_discretization} supplies a deterministic bound $FR_L(f)\le r_L$ from smoothness, so stable recovery follows for bandlimited signals without any a priori sparsity assumption on $\widehat f$.
\end{remark}

\begin{remark}[Spherical sampling viewpoint]
Theorem \ref{thm:sphere_missing_values} is the spherical analogue of Theorem \ref{thm:missing-values-unit-square}. The unknown is a global field on $S^2$ and the measurements are point evaluations at randomly chosen locations, modeling incomplete directional or geophysical data on a spherical geometry. The theorem states that if the signal is bandlimited and has mild regularity, then an explicit Fourier ratio bound follows deterministically, and this bound is sufficient to guarantee stable recovery from incomplete samples by standard $\ell^1$ minimization of spherical harmonic coefficients. In this way, regularity replaces sparsity as the structural input needed for compressed sensing type recovery on the sphere.
\end{remark}

\subsection{Data distribution on the unit sphere}

In many applications, the underlying data live on a curved surface rather than on a Euclidean grid. A basic model is a function on the unit sphere $S^2\subset{\mathbb R}^3$. In this setting it is natural to represent the signal in the spherical harmonic basis and to treat bandlimiting as the analogue of having a finite bandwidth in the Euclidean Fourier transform \cite{RW12}.

Fix a positive integer $L$, and let $V_L$ denote the space of spherical harmonics of degree at most $L$. We think of $f\in V_L$ as the unknown signal. As in the square case, we imagine that only a randomly chosen collection of point evaluations of $f$ is available, and the rest of the values are missing. The goal is to recover $f$ from these random samples by solving an $\ell^1$ minimization problem for the spherical harmonic coefficients, as in spherical harmonic compressed sensing \cite{RW12}.

The key step is to control a spherical Fourier ratio. Proposition \ref{prop:sphere_discretization} shows that if $f$ also has mild smoothness, for example $f\in C^2(S^2)$, then one obtains an explicit upper bound on the spherical Fourier ratio in terms of $\|f\|_{C^2(S^2)}$, $\|f\|_{L^2(S^2)}$, and $L$. This replaces the discretization estimate in the square setting. Once this bound is available, the spherical compressed sensing theorem implies that a random set of point samples of size comparable to
$$
\frac{r_L^2}{\eps^2}\log\left(\frac{r_L}{\eps}\right)^2\log\left((L+1)^2\right)
$$
is enough to guarantee stable recovery in $L^2(S^2)$ by $\ell^1$ minimization, with high probability \cite{RW12,FR13}.

Our next result makes this precise and gives a missing values recovery theorem on $S^2$ which is directly parallel to Theorem \ref{thm:missing-values-unit-square}.

\begin{theorem} \label{thm:sphere_missing_values}
Let $L$ be a positive integer and let $V_L$ denote the space of spherical harmonics of degree at most $L$ on $S^2$. Let $f$ be a real-valued function in $V_L$ which also belongs to $C^2(S^2)$. Set
$$
A_L=
\frac{\left|\int_{S^2} f(\omega)d\sigma(\omega)\right|}{\|f\|_{L^2(S^2)}}.
$$
Set
$$
B_L=
C_0
\frac{\|f\|_{C^2(S^2)}}{\|f\|_{L^2(S^2)}}
\log L.
$$
Set
$$
C_L=
C_0
\frac{\|f\|_{C^2(S^2)}}{\|f\|_{L^2(S^2)}}
\frac{1}{L}.
$$
Define
$$
r_L=A_L+B_L+C_L.
$$
Fix
$$
\eps\in (0,1/2).
$$
By Proposition \ref{prop:sphere_discretization}, one has
$$
FR_L(f)\le r_L,
$$
where $C_0$ is the absolute constant appearing there.

Define
$$
D=(L+1)^2.
$$
Sample points $\omega_1,\dots,\omega_q$ independently and uniformly from $S^2$, where
$$
q=
C
\frac{r_L^2}{\eps^2}
\log\left(\frac{r_L}{\eps}\right)^2
\log D.
$$
Define the sampled multiset
$$
\Omega=(\omega_1,\dots,\omega_q).
$$
Define the empirical norm on the multiset by
$$
\|g\|_{L^2(\Omega)}=
\left(\sum_{j=1}^q|g(\omega_j)|^2\right)^{1/2}.
$$
Assume that the values $f(\omega_j)$ are observed for $j=1,\dots,q$ and missing elsewhere on the sphere. Let $f^*\in V_L$ be a solution to the convex program
$$
\min_{h\in V_L}\|\widehat h\|_1
$$
subject to the constraint
$$
\|f-h\|_{L^2(\Omega)}\le \eps \|f\|_{L^2(S^2)}.
$$
If $C$ is a sufficiently large universal constant, then with high probability
$$
\|f^*-f\|_{L^2(S^2)}\le 11.47\eps\|f\|_{L^2(S^2)}.
$$
In particular, a bandlimited spherical signal in $V_L$ with the explicit Fourier ratio bound supplied by Proposition \ref{prop:sphere_discretization} can be recovered from random point samples by the $\ell^1$ minimization procedure in Theorem \ref{thm:sphere_121} \cite{RW12}.
\end{theorem}

\begin{remark}[Spherical Fourier ratio interpretation]
Theorem \ref{thm:sphere_missing_values} is the spherical analogue of Theorem \ref{thm:missing-values-unit-square}. The result starts from a bandlimited function on the sphere, modeled by membership in $V_L$, and interprets its point evaluations as measurements taken by sensors distributed randomly on $S^2$.

The role of regularity is to supply a deterministic bound on the spherical Fourier ratio. Proposition \ref{prop:sphere_discretization} shows that if $f\in C^2(S^2)$ then one has an explicit estimate of the form
$$
FR_L(f)\le r_L.
$$
This estimate depends on $\|f\|_{C^2(S^2)}$, $\|f\|_{L^2(S^2)}$, and the bandwidth parameter $L$. It does not assume any sparsity of $f$ in the spatial variable.

Once the Fourier ratio is controlled, Theorem \ref{thm:sphere_121} implies that $f$ can be recovered stably from a random collection of its point samples by minimizing the $\ell^1$ norm of its spherical harmonic coefficients. Theorem \ref{thm:sphere_missing_values} combines these two ingredients to give a missing values recovery theorem on $S^2$ with an explicit sample complexity and an explicit relative $L^2(S^2)$ error bound. This is closely aligned with the spherical harmonic compressed sensing literature \cite{RW12} and the general stability theory for $\ell^1$ recovery \cite{CRT06,FR13}.

From an applied perspective, this provides a rigorous justification for imputing missing spherical measurements using compressed sensing techniques \cite{CRT06,FR13,RW12}. The guiding principle is that regularity forces a quantitative form of compressibility of the spherical harmonic expansion, expressed through the Fourier ratio, and this is enough to guarantee stable recovery from random samples.
\end{remark}

\begin{figure}[h]
\centering
\begin{tikzpicture}[
  node distance=2.15cm,
  box/.style={draw, rounded corners, align=center,
              minimum width=8.4cm, text width=8.0cm, inner sep=7pt}
]
\node (a) [box]
{Spherical data (imaging viewpoint) \\ underlying signal $f$ on $S^2$ \\ $f\in V_L\cap C^2(S^2)$};
\node (b) [box, below of=a]
{Spectral representation \\ spherical harmonic expansion \\ coefficients in $\{Y_\ell^m\}_{0\le \ell\le L}$};
\node (c) [box, below of=b]
{Deterministic structure \\ regularity and bandlimiting imply \\ $FR_L(f)\le r_L$ (Proposition \ref{prop:sphere_discretization})};
\node (d) [box, below of=c]
{Incomplete measurements \\ random sample multiset $\Omega$ on $S^2$ \\ $q=C\frac{r_L^2}{\eps^2}\log(r_L/\eps)^2\log((L+1)^2)$};
\node (e) [box, below of=d]
{Reconstruction (inverse problem) \\ recover $f$ from point samples by convex optimization \\ minimize $\|\widehat h\|_1$ subject to data fidelity};
\node (f) [box, below of=e]
{Stability guarantee \\ $\|f^*-f\|_{L^2(S^2)}\le 11.47\eps\|f\|_{L^2(S^2)}$ \\ (Theorem \ref{thm:sphere_121})};
\draw[->] (a) -- (b);
\draw[->] (b) -- (c);
\draw[->] (c) -- (d);
\draw[->] (d) -- (e);
\draw[->] (e) -- (f);
\end{tikzpicture}
\caption{Reconstruction pipeline for spherical inverse problems. A bandlimited and smooth signal on the sphere is sampled at random locations, regularity yields an explicit spherical Fourier ratio bound, and stable recovery from incomplete point measurements follows from $\ell^1$ minimization of spherical harmonic coefficients.}
\end{figure}
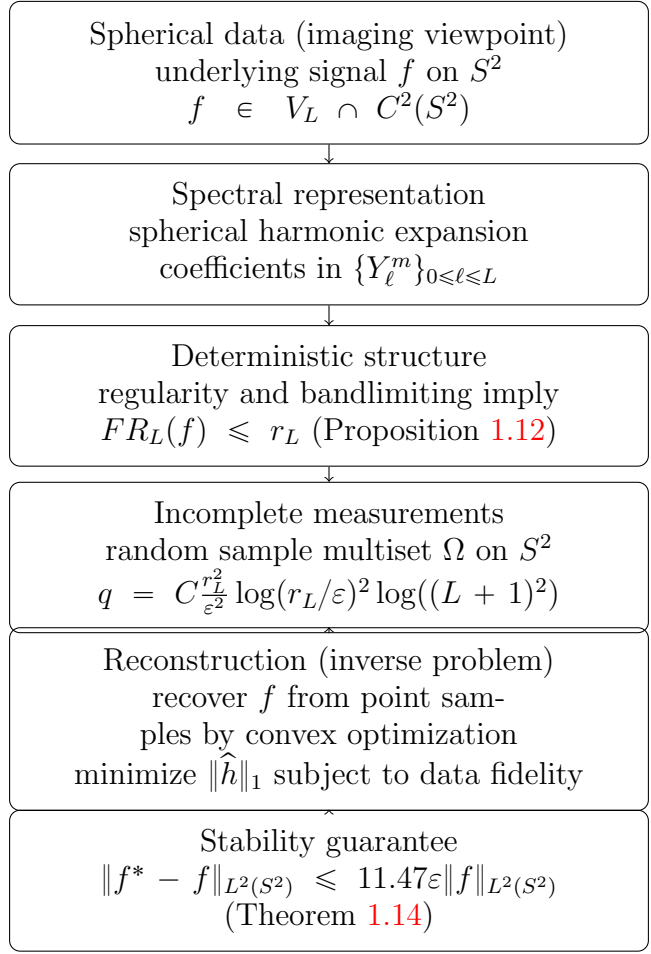

\begin{proposition} \label{prop:sphere_discretization}
Let $L$ be a positive integer and let $V_L$ denote the space of spherical harmonics of degree at most $L$ on $S^2$. Let $f$ be a real-valued function in $V_L\cap C^2(S^2)$. Let $\{Y_\ell^m\}_{0\le \ell\le L,-\ell\le m\le \ell}$ denote the standard orthonormal basis of spherical harmonics on $S^2$ in $L^2(S^2)$.

Suppose that $\{\omega_j,w_j\}_{j=1}^M$ is a quadrature rule on $S^2$ which is exact for spherical harmonics of degree at most $2L$, in the sense that
$$
\int_{S^2} P(\omega)d\sigma(\omega)=\sum_{j=1}^M w_j P(\omega_j)
$$
for every spherical harmonic $P$ of degree at most $2L$ \cite{SW04,BRV13}. Define the sampled function $g:\{1,\dots,M\}\to{\mathbb R}$ by
$$
g(j)=f(\omega_j), \qquad j=1,\dots,M.
$$
For $0\le \ell\le L$ and $-\ell\le m\le \ell$, define the discrete coefficients by
$$
\widehat g(\ell,m)=\sum_{j=1}^M w_j g(j)Y_\ell^m(\omega_j).
$$
Also define the continuous spherical harmonic coefficients by
$$
\widehat f(\ell,m)=\int_{S^2} f(\omega)Y_\ell^m(\omega)d\sigma(\omega).
$$
Since $f\in V_L$, the product $fY_\ell^m$ is a spherical polynomial of degree at most $2L$, so the exactness hypothesis applies with $P=fY_\ell^m$. Consequently, for every $0\le \ell\le L$ and $-\ell\le m\le \ell$ one has
$$
\widehat g(\ell,m)=\widehat f(\ell,m).
$$
In particular Parseval's identity holds in the form
$$
\sum_{\ell=0}^L\sum_{m=-\ell}^{\ell}|\widehat g(\ell,m)|^2
=
\sum_{\ell=0}^L\sum_{m=-\ell}^{\ell}|\widehat f(\ell,m)|^2
=
\|f\|_{L^2(S^2)}^2,
$$
and the discrete energy satisfies
$$
\sum_{j=1}^M w_j|g(j)|^2=\|f\|_{L^2(S^2)}^2.
$$
Define
$$
FR_L(f)=\frac{\|\widehat f\|_1}{\|\widehat f\|_2}.
$$
Then
$$
FR_L(f)=\frac{\|\widehat g\|_1}{\|\widehat g\|_2}.
$$
Moreover, one has
$$
FR_L(f)\le
\frac{\left|\int_{S^2} f(\omega)d\sigma(\omega)\right|}{\|f\|_{L^2(S^2)}}
+
C_0\frac{\|f\|_{C^2(S^2)}}{\|f\|_{L^2(S^2)}}\log L
+
C_0\frac{\|f\|_{C^2(S^2)}}{\|f\|_{L^2(S^2)}}\frac{1}{L},
$$
where one may take
$$
C_0=48\sqrt{\pi}.
$$
In particular, if $f\ge 0$ and
$$
\mu=\int_{S^2} f(\omega)d\sigma(\omega)>0,
$$
then
$$
FR_L(f)\le
\sqrt{4\pi}
+
C_0\sqrt{4\pi}\frac{\|f\|_{C^2(S^2)}}{\mu}\log L
+
C_0\sqrt{4\pi}\frac{\|f\|_{C^2(S^2)}}{\mu}\frac{1}{L}.
$$
\end{proposition}

\begin{remark}
The proof of Proposition~1.12 follows the same strategy as Proposition~1.6. One uses the exact quadrature rule to identify discrete and continuous spherical harmonic coefficients up to degree $L$, applies Parseval’s identity on $V_L$, separates the mean mode $\ell=0$, and estimates the remaining coefficients by integration by parts using the eigenvalue equation for spherical harmonics together with the $C^2(S^2)$ regularity of $f$. The resulting logarithmic growth arises from summing the explicit $\ell^{-1}$ decay of the coefficients over $1 \le \ell \le L$, and all constants, including the value $C_0=48\sqrt{\pi}$, can be tracked explicitly.
\end{remark}

\begin{theorem}[Spherical variant of Theorem \ref{thm:ZNd_recovery}] \label{thm:sphere_121}
Let $L$ be a positive integer and let $V_L$ denote the space of spherical harmonics of degree at most $L$ on $S^2$. Let $\{Y_\ell^m\}_{0\le \ell\le L, -\ell\le m\le \ell}$ be an orthonormal basis for $V_L$ in $L^2(S^2)$. For $f\in V_L$ define its spherical Fourier coefficients by
$$
\widehat f(\ell,m)=\int_{S^2} f(\omega)Y_\ell^m(\omega)d\sigma(\omega),
$$
and define
$$
FR_L(f)=\frac{\|\widehat f\|_1}{\|\widehat f\|_2},
$$
where the $\ell^1$ and $\ell^2$ norms are taken over the index set $\{(\ell,m):0\le \ell\le L,-\ell\le m\le \ell\}$. Fix parameters $\eps\in(0,1/2)$ and $r\ge 1$. Let $D=(L+1)^2$. Sample points $\omega_1,\dots,\omega_q$ independently and uniformly from $S^2$, where
$$
q=C\frac{r^2}{\eps^2}\log(r/\eps)^2\log D.
$$
Define the sampled multiset $\Omega=(\omega_1,\dots,\omega_q)$ and define the empirical norm
$$
\|g\|_{L^2(\Omega)}=\left(\sum_{j=1}^q|g(\omega_j)|^2\right)^{1/2}.
$$
Assume that $f\in V_L$ satisfies $FR_L(f)\le r$. Let $f^*\in V_L$ be a solution to the convex program
$$
\min_{h\in V_L}\|\widehat h\|_1 \ \ \text{subject to}\ \ \|f-h\|_{L^2(\Omega)}\le \eps\|f\|_{L^2(S^2)}.
$$
If $C$ is a sufficiently large universal constant, then with high probability,
$$
\|f^*-f\|_{L^2(S^2)}\le 11.47 \, \eps\|f\|_{L^2(S^2)}.
$$
\end{theorem}

\section{Proofs of the main results}

\subsection{Proof of Proposition \ref{prop:discretization}}

\begin{lemma} \label{lem:riemann_sum_C1}
Let $h\in C^1([0,1]^2)$ be 1-periodic in each variable. Then for every integer $N\ge 2$,
$$
\left|\frac{1}{N^2}\sum_{x\in{\mathbb Z}_N^2} h(x_1/N,x_2/N)-\int_{[0,1]^2}h(u)du\right|
\le \frac{1}{N}\left(\left\|\frac{\partial h}{\partial x_1}\right\|_{L^\infty([0,1]^2)}+\left\|\frac{\partial h}{\partial x_2}\right\|_{L^\infty([0,1]^2)}\right).
$$
\end{lemma}

\begin{proof}
Write the difference between the Riemann sum and the integral as an average over the $N^2$ squares
$$
Q_{k,\ell}=\left[\frac{k}{N},\frac{k+1}{N}\right]\times\left[\frac{\ell}{N},\frac{\ell+1}{N}\right]
$$
with $0\le k,\ell\le N-1$. For $(u_1,u_2)\in Q_{k,\ell}$, the triangle inequality and the mean value theorem give
$$
|h(k/N,\ell/N)-h(u_1,u_2)|
\le \frac{1}{N}\left\|\frac{\partial h}{\partial x_1}\right\|_\infty+\frac{1}{N}\left\|\frac{\partial h}{\partial x_2}\right\|_\infty.
$$
Integrate this inequality over each $Q_{k,\ell}$, sum over $k,\ell$, and divide by $N^2$ to obtain the claimed bound.
\end{proof}

\begin{lemma} \label{lem:L2_lower_bound_g}
Let $f\in C^2([0,1]^2)$ be 1-periodic in each variable and let $g(x_1,x_2)=f(x_1/N,x_2/N)$ on ${\mathbb Z}_N^2$. Then for every $N\ge 2$,
$$
\|g\|_{L^2({\mathbb Z}_N^2)}^2\ge N^2\|f\|_{L^2([0,1]^2)}^2-4N\|f\|_{C^2([0,1]^2)}^2.
$$
In particular, if
$$
N\|f\|_{L^2([0,1]^2)}^2\ge 8\|f\|_{C^2([0,1]^2)}^2,
$$
then
$$
\|g\|_{L^2({\mathbb Z}_N^2)}\ge \frac{N}{\sqrt{2}}\|f\|_{L^2([0,1]^2)}.
$$
\end{lemma}

\begin{proof}
Apply Lemma \ref{lem:riemann_sum_C1} with $h(u)=|f(u)|^2$. Since for $j=1,2$ one has
$$
\left\|\frac{\partial}{\partial x_j}|f(u)|^2\right\|_\infty
=\left\|2\operatorname{Re}\left(\overline{f(u)}\partial_{x_j}f(u)\right)\right\|_\infty
\le 2 \|f\|_\infty \|\partial_{x_j}f\|_\infty
\le 2\|f\|_{C^2([0,1]^2)}^2,
$$
Lemma \ref{lem:riemann_sum_C1} gives
$$
\left|\frac{1}{N^2}\sum_{x\in{\mathbb Z}_N^2}|g(x)|^2-\int_{[0,1]^2}|f(u)|^2du\right|
\le \frac{4}{N}\|f\|_{C^2([0,1]^2)}^2.
$$
Hence
$$
\sum_{x\in{\mathbb Z}_N^2}|g(x)|^2
\ge N^2\|f\|_{L^2([0,1]^2)}^2-4N\|f\|_{C^2([0,1]^2)}^2.
$$
This is the first inequality. Under the displayed hypothesis,
$$
N^2\|f\|_{L^2([0,1]^2)}^2-4N\|f\|_{C^2([0,1]^2)}^2
\ge \frac{1}{2}N^2\|f\|_{L^2([0,1]^2)}^2,
$$
and taking square-roots yields the second inequality.
\end{proof}

\begin{proof}[Proof of Proposition \ref{prop:discretization}]
With the normalization
$$
\widehat g(m)=\frac{1}{N}\sum_{x\in{\mathbb Z}_N^2}\chi(-x\cdot m)g(x),
$$
Parseval gives
$$
\|\widehat g\|_2=\|g\|_{L^2({\mathbb Z}_N^2)}.
$$

We first bound the zero frequency. Lemma \ref{lem:riemann_sum_C1} applied to $h=f$ gives
$$
\left|\frac{1}{N^2}\sum_{x\in{\mathbb Z}_N^2}g(x)-\int_{[0,1]^2}f(u)du\right|
\le \frac{2}{N}\|f\|_{C^2([0,1]^2)}.
$$
Multiplying by $N$ and using $\widehat g(0)=\frac{1}{N}\sum_x g(x)$ yields
$$
|\widehat g(0)|
\le
N\left|\int_{[0,1]^2} f(u)du\right|+2\|f\|_{C^2([0,1]^2)}.
$$

Next consider $m\neq 0$. A discrete summation-by-parts argument in a nonzero coordinate direction gives
$$
|\widehat g(m)|\le \frac{2\pi^2 N}{|m|^2}\|f\|_{C^2([0,1]^2)}.
$$

\begin{remark}[Justification of the $|m|^{-2}$ decay bound]
We justify the estimate
\[
|\widehat{g}(m)|\leqslant \frac{2\pi^{2}N}{|m|^{2}}\| f\|_{C^{2}([0,1]^{2})}
\]
used above. The argument is based on discrete summation by parts applied twice in a coordinate direction in which $m$ is nonzero, with the frequency magnitude understood in the wrapped sense on $\mathbb{Z}_N$.

For $k \in \mathbb{Z}_N$ define
\[
|k|_{\ast} = \min \{k,N - k\},
\]
and for $m = (m_{1},m_{2})\in \mathbb{Z}_{N}^{2}$ define
\[
|m|_{\ast}^{2} = |m_{1}|_{\ast}^{2} + |m_{2}|_{\ast}^{2}.
\]
If $m\neq 0$, then at least one of the coordinates $m_{1},m_{2}$ is nonzero. 
Choose the coordinate $j$ for which $|m_j|_*$ is maximal; by symmetry we may assume it is $m_1$.
Then $|m_1|_*^2 \geq \frac12 |m|_*^2$.

Using the identity
\[
|1 - \chi (-m_1)| = 2\left|\sin \left(\frac{\pi m_1}{N}\right)\right|,
\]
we obtain the lower bound
\[
|1 - \chi (-m_1)|\geqslant \frac{4}{N} |m_1|_{\ast}.
\]
Write
\[
\sum_{x_1\in \mathbb{Z}_N}\chi (-x_1m_1)g(x_1,x_2) 
    = \frac{1}{1 - \chi(-m_1)}\sum_{x_1\in \mathbb{Z}_N}\chi (-(x_1+1)m_1)\Delta_1g(x_1,x_2),
\]
where $\Delta_{1}g(x_{1},x_{2}) = g(x_{1}+1,x_{2}) - g(x_{1},x_{2})$.
Applying the same identity once more to $\Delta_{1}g$ yields
\[
\sum_{x_1\in \mathbb{Z}_N}\chi (-x_1m_1)g(x_1,x_2) 
    = \frac{1}{(1 - \chi(-m_1))^2}\sum_{x_1\in \mathbb{Z}_N}\chi (-(x_1+2)m_1)\Delta_1^2 g(x_1,x_2).
\]
Taking absolute values and summing in $x_{2}$, we obtain
\[
|\widehat{g}(m)|\leqslant \frac{N}{|1 - \chi(-m_1)|^2}\| \Delta_1^2 g\|_{\infty}.
\]
Since $g(x_{1},x_{2}) = f(x_{1} / N,x_{2} / N)$ and $f\in C^{2}([0,1]^{2})$, one has
\[
\| \Delta_1^2 g\|_{\infty}\leqslant \frac{1}{N^2}\| f\|_{C^2 ([0,1]^2)}.
\]
Combining these estimates gives
\[
|\widehat{g}(m)|\leqslant C\frac{N}{|m_1|_{\ast}^2}\| f\|_{C^2 ([0,1]^2)},
\]
where $C$ is an absolute constant, which may be taken to be $2\pi^2$. By switching the roles of $m_1$ and $m_2$, depending on which is larger we can conclude
\[
|\widehat{g}(m)|\leqslant 2C\frac{N}{|m|_{\ast}^2}\| f\|_{C^2 ([0,1]^2)}.
\]

Finally, since replacing the Euclidean magnitude $|m|$ by the wrapped magnitude $|m|_{\ast}$ only affects counting estimates by absolute constants, we may harmlessly continue to write $|m|$ in the sequel. In particular, one has
\[
\sum_{m\neq 0}\frac{1}{|m|_{\ast}^{2}}\leqslant C\log N,
\]
which yields the logarithmic growth appearing in the Fourier ratio bound.
\end{remark}

Consequently,
$$
\sum_{m\neq 0}|\widehat g(m)|
\le 2\pi^2 N\|f\|_{C^2([0,1]^2)}\sum_{m\neq 0}\frac{1}{|m|^2}.
$$
Using the estimate
$$
\sum_{m\neq 0}\frac{1}{|m|^2}\le 2\log N+2,
$$
we obtain
$$
\sum_{m\neq 0}|\widehat g(m)|
\le 4\pi^2 N\|f\|_{C^2([0,1]^2)}\log N+4\pi^2 N\|f\|_{C^2([0,1]^2)}.
$$
Combining the zero and nonzero bounds gives
$$
\|\widehat g\|_1
\le
N\left|\int_{[0,1]^2} f(u)du\right|
+4\pi^2 N\|f\|_{C^2([0,1]^2)}\log N
+4\pi^2 N\|f\|_{C^2([0,1]^2)}
+2\|f\|_{C^2([0,1]^2)}.
$$

By the hypothesis of Proposition \ref{prop:discretization}, Lemma \ref{lem:L2_lower_bound_g} implies
$$
\|\widehat g\|_2=\|g\|_{L^2({\mathbb Z}_N^2)}\ge \frac{N}{\sqrt{2}}\|f\|_{L^2([0,1]^2)}.
$$
Dividing the bound for $\|\widehat g\|_1$ by this lower bound and using $\sqrt{2}\le 2$ yields
\begin{align*}
FR(g)
&\le
2\frac{\left|\int_{[0,1]^2} f(u)du\right|}{\|f\|_{L^2([0,1]^2)}}
+8\pi^2\frac{\|f\|_{C^2([0,1]^2)}}{\|f\|_{L^2([0,1]^2)}}\log N
+\frac{8\pi^2}{\|f\|_{L^2([0,1]^2)}}\|f\|_{C^2([0,1]^2)}
+\frac{4}{N}\frac{\|f\|_{C^2([0,1]^2)}}{\|f\|_{L^2([0,1]^2)}}.
\end{align*}

Since $N \ge 2$, one has $\log N \ge \log 2$. Therefore
$$
8\pi^2 \frac{\|f\|_{C^2([0,1]^2)}}{\|f\|_{L^2([0,1]^2)}} 
= \frac{8\pi^2}{\log 2}\frac{\|f\|_{C^2([0,1]^2)}}{\|f\|_{L^2([0,1]^2)}} \log 2
\le \frac{8\pi^2}{\log 2}\frac{\|f\|_{C^2([0,1]^2)}}{\|f\|_{L^2([0,1]^2)}} \log N.
$$
Since $\frac{1}{\log 2} < 2$, this implies
$$
8\pi^2 \frac{\|f\|_{C^2([0,1]^2)}}{\|f\|_{L^2([0,1]^2)}}
\le 16\pi^2 \frac{\|f\|_{C^2([0,1]^2)}}{\|f\|_{L^2([0,1]^2)}} \log N.
$$
In addition, since $8\pi^2 \ge 4$, we have
$$
\frac{4}{N}\frac{\|f\|_{C^2([0,1]^2)}}{\|f\|_{L^2([0,1]^2)}}
\le \frac{8\pi^2}{N}\frac{\|f\|_{C^2([0,1]^2)}}{\|f\|_{L^2([0,1]^2)}}.
$$
Substituting these two bounds into the previous display yields the claimed estimate.
\end{proof}

\subsection{Proof of Theorem \ref{thm:ZNd_recovery}}

\begin{proof}
The argument follows the same blueprint as the corresponding recovery theorem in \cite{A2025}. The proof uses three inputs: the approximation statement that converts a Fourier ratio hypothesis into a sparse Fourier proxy, a restricted isometry estimate for random sampling of characters, and the standard implication from restricted isometry to stable $\ell^1$ recovery \cite{CRT06,FR13}.

Write
$$
f(x)=\frac{1}{N^{\frac{d}{2}}}\sum_{m\in{\mathbb Z}_N^d}\chi(x\cdot m)\widehat f(m).
$$
The unknown object is the coefficient vector $\widehat f$ in ${\mathbb C}^D$. The Fourier ratio assumption $FR(f)\le r$ is a numerical sparsity hypothesis on $\widehat f$ because
$$
\frac{\|\widehat f\|_1^2}{\|\widehat f\|_2^2}\le r^2.
$$

By Theorem 1.15 in \cite{A2025}, for every $\eps\in(0,1/2)$ there exists a set $\Omega\subset{\mathbb Z}_N^d$ with
$$
|\Omega|\le r^2\eps^{-2}
$$
such that the truncated Fourier series
$$
P(x)=\frac{1}{N^{\frac{d}{2}}}\sum_{m\in\Omega}\chi(x\cdot m)\widehat f(m)
$$
satisfies
$$
\|f-P\|_{L^2({\mathbb Z}_N^d)}\le \eps\|f\|_{L^2({\mathbb Z}_N^d)}.
$$
Thus $f$ is within relative error $\eps$ of a function $P$ whose Fourier transform is supported on at most $S=r^2\eps^{-2}$ frequencies.

We now proceed as in the proof of Theorem 1.21 in \cite{A2025}. Let $X\subset{\mathbb Z}_N^d$ be a random subset of cardinality $q$ as in the theorem. Define the sampling operator $T_X$ by $T_X g=g|_X$. If $a\in{\mathbb C}^D$ is a Fourier coefficient vector supported on $\Omega$, let $F_\Omega a$ denote the corresponding function on ${\mathbb Z}_N^d$ given by the inverse Fourier series
$$
(F_\Omega a)(x)=\frac{1}{N^{\frac{d}{2}}}\sum_{m\in\Omega}\chi(x\cdot m)a(m).
$$
Then the restriction of $F_\Omega a$ to $X$ can be written as a matrix-vector product $A_\Omega a$, where $A_\Omega$ is the $q\times S$ matrix whose rows are indexed by $x\in X$, whose columns are indexed by $m\in\Omega$, and whose entries are
$$
(A_\Omega)_{x,m}=\frac{1}{N^{\frac{d}{2}}}\chi(x\cdot m).
$$
The crucial input is that, for $q=C S\log(S)^2\log D$ with $C$ sufficiently large, the random matrix $A_\Omega$ satisfies the restricted isometry condition of the order required in the recovery theorem with high probability. This is the same type of restricted isometry statement used for random sampling from bounded orthogonal systems \cite{Bourgain89,Talagrand98,RV08,Rau10}.

Once the restricted isometry condition holds, the standard compressed sensing stability theorem (see the proof of Theorem 1.21 in \cite{A2025}) implies that the solution $\widehat f^*$ to the $\ell^1$ minimization problem yields the estimate
$$
\|\widehat f^*-\widehat f\|_2\le 11.47\, \eps\|f\|_{L^2({\mathbb Z}_N^d)},
$$
and therefore, by Parseval and the Fourier inversion formula, the reconstructed function $f^*$ satisfies
$$
\|f^*-f\|_{L^2({\mathbb Z}_N^d)}\le 11.47\, \eps\|f\|_{L^2({\mathbb Z}_N^d)}.
$$
This gives the claimed estimate.
\end{proof}

\subsection{Proof of Theorem \ref{thm:missing-values-unit-square}}

\begin{proof}
The proof is an application of Proposition \ref{prop:discretization} and Theorem \ref{thm:ZNd_recovery}.

Define $g:{\mathbb Z}_N^2\to{\mathbb R}$ by
$$
g(x_1,x_2)=f(x_1/N,x_2/N).
$$
Proposition \ref{prop:discretization} applies to $f$ and $N$ and yields the bound
$$
FR(g)\le r_N.
$$
Let
$$
D=N^2.
$$
Fix $\eps\in(0,1/2)$ and let $X\subset{\mathbb Z}_N^2$ be chosen uniformly at random among all subsets of cardinality
$$
|X|=C\frac{r_N^2}{\eps^2}\log(r_N/\eps)^2\log D.
$$
Define the empirical norm by
$$
\|h\|_{L^2(X)}=\left(\sum_{x\in X}|h(x)|^2\right)^{1/2}.
$$
Assume that the values $g(x)$ are known for $x\in X$ and unknown for $x\in{\mathbb Z}_N^2\setminus X$. Let $g^*:{\mathbb Z}_N^2\to{\mathbb R}$ be a solution to the convex optimization problem
$$
\min_{h:{\mathbb Z}_N^2\to{\mathbb R}}{\|\widehat h\|}_1,
$$
subject to the constraint
$$
\|g-h\|_{L^2(X)}\le\eps\|g\|_{L^2({\mathbb Z}_N^2)}.
$$
Since $FR(g)\le r_N$, all of the hypotheses of Theorem \ref{thm:ZNd_recovery} are satisfied with $d=2$ and $r=r_N$. Therefore, if $C$ is sufficiently large, Theorem \ref{thm:ZNd_recovery} implies that with high probability
$$
\|g^*-g\|_{L^2({\mathbb Z}_N^2)}\le 11.47\eps\|g\|_{L^2({\mathbb Z}_N^2)}.
$$
This is exactly the desired bound. The concluding statement about missing values follows because $g^*$ is defined on all of ${\mathbb Z}_N^2$ and agrees with the observed data on $X$ up to the stated tolerance.
\end{proof}

\subsection{Proof of Theorem \ref{thm:sphere_missing_values}}

\begin{proof}
The proof is an application of Proposition \ref{prop:sphere_discretization} and Theorem \ref{thm:sphere_121}.

Let $L$ be fixed and let $f$ be as in the statement. Define $A_L$, $B_L$, $C_L$, and $r_L$ as in the theorem. Proposition \ref{prop:sphere_discretization} applies to $f$ and yields the bound
$$
FR_L(f)\le r_L.
$$
Fix $\eps\in(0,1/2)$ and define
$$
D=(L+1)^2.
$$
Sample points $\omega_1,\dots,\omega_q$ independently and uniformly from $S^2$, where
$$
q=
C
\frac{r_L^2}{\eps^2}
\log\left(\frac{r_L}{\eps}\right)^2
\log D.
$$
Define the sampled multiset $\Omega=(\omega_1,\dots,\omega_q)$ and the empirical norm
$$
\|g\|_{L^2(\Omega)}=
\left(\sum_{j=1}^q|g(\omega_j)|^2\right)^{1/2}.
$$
Let $f^*\in V_L$ be a solution to the convex optimization problem
$$
\min_{h\in V_L}\|\widehat h\|_1
$$
subject to the constraint
$$
\|f-h\|_{L^2(\Omega)}\le \eps \|f\|_{L^2(S^2)}.
$$
Since $FR_L(f)\le r_L$, all of the hypotheses of Theorem \ref{thm:sphere_121} are satisfied with $r=r_L$ and the same value of $\eps$. Therefore, if $C$ is sufficiently large, Theorem \ref{thm:sphere_121} implies that with high probability
$$
\|f^*-f\|_{L^2(S^2)}\le 11.47\eps\|f\|_{L^2(S^2)}.
$$
This is exactly the desired estimate. The concluding statement about missing values follows because $f^*$ is defined on $S^2$ and provides values at every point of the sphere.
\end{proof}

\subsection{Proof of Theorem \ref{thm:sphere_121}}

\begin{proof}
The proof follows the same blueprint as the corresponding recovery theorem in \cite{A2025}. The argument uses two inputs beyond the approximation statement based on the Fourier ratio. The first input is that the measurement operator obtained by restricting a function to a random multiset of sample points can be represented as a random submatrix of an orthogonal system, and the second input is that this random submatrix satisfies a restricted isometry condition of the order needed for sparse recovery. Once these two facts are available, the remainder of the argument is the standard implication from a restricted isometry condition to stable recovery by $\ell^1$ minimization, and the numerical constant $11.47$ comes from the same stability estimate used there, in line with standard compressed sensing stability theory \cite{CRT06,FR13}.

In the spherical setting, the unknown object is the coefficient vector of $f$ in the orthonormal basis $\{Y_\ell^m\}$, which has dimension $D=(L+1)^2$. Writing $h(\omega)=\sum_{\ell,m}a_{\ell,m}Y_\ell^m(\omega)$ identifies $h\in V_L$ with a vector $a\in{\mathbb C}^D$, and the constraint $\|f-h\|_{L^2(\Omega)}\le \eps\|f\|_{L^2(S^2)}$ is exactly a constraint on the sampled values of the difference, hence a constraint of the form $\|A(a-a_0)\|_2\le \eps\|f\|_{L^2(S^2)}$, where $A$ is the sampling matrix with entries $A_{j,(\ell,m)}=Y_\ell^m(\omega_j)$. The objective $\|\widehat h\|_1$ is the $\ell^1$ norm of this coefficient vector, so the optimization problem in Theorem \ref{thm:sphere_121} is the same $\ell^1$ recovery problem as in \cite{A2025,FR13}.

The missing ingredient is the spherical analogue of the restricted isometry estimate for the sampling matrix $A$ when the sampling points are chosen independently and uniformly on $S^2$. This is available in the literature for spherical harmonics up to degree $L$ and implies that, for $q$ as in Theorem \ref{thm:sphere_121}, the matrix $A$ satisfies the restricted isometry condition of the order required in the recovery theorem with high probability, after the same harmless normalizations that appear in the spherical harmonic compressed sensing literature. One convenient reference is Rauhut and Ward \cite{RW12}. With this restricted isometry input in hand, the proof is obtained by repeating the standard restricted isometry to $\ell^1$ recovery argument in this spherical harmonic setting.
\end{proof}

\section*{Discussion and illustrative examples}

The results presented in this paper establish a deterministic link between regularity and compressibility through the Fourier ratio, providing a theoretical foundation for stable recovery from incomplete samples without imposing sparsity assumptions. To illustrate the practical implications of our bounds, we offer some qualitative examples and remarks.

\begin{example}[Scaling behavior of $r_N$]
Consider a fixed $C^2$ function $f$ on $[0,1]^2$ with $\|f\|_{L^2([0,1]^2)}\approx 1$ and $\|f\|_{C^2([0,1]^2)}\approx 1$. The bound in Proposition \ref{prop:discretization} becomes
$$
r_N \approx A_N +16\pi^2 \log N + \frac{8\pi^2}{N},
$$
where $A_N = 2\left|\int f\right|/\|f\|_{L^2}$. For large $N$, the dominant term is $O(\log N)$. Thus the Fourier ratio grows only logarithmically with the grid resolution, implying that the required number of random samples for stable recovery scales like $O(\log^3 N)$ up to additional logarithmic factors already present in the sampling bound.
\end{example}

\begin{example}[Spherical setting with fixed smoothness]
For a bandlimited spherical function $f \in V_L \cap C^2(S^2)$ with $\|f\|_{L^2(S^2)}\approx 1$ and $\|f\|_{C^2(S^2)}\approx 1$, Proposition \ref{prop:sphere_discretization} yields
$$
r_L \approx A_L + C_0 \log L + \frac{C_0}{L}.
$$
Again the logarithmic term dominates for large bandwidth $L$. The sample complexity for stable recovery then scales polylogarithmically in $L$, while the dimension of $V_L$ is $(L+1)^2$. The gap between these quantities illustrates the compressibility induced by smoothness.
\end{example}

\begin{remark}[Practical interpretation of the Fourier ratio]
The Fourier ratio $FR(f) = \|\widehat{f}\|_1/\|\widehat{f}\|_2$ can be viewed as a measure of numerical sparsity. If the Fourier coefficients are concentrated on a small set, the $\ell^1$ norm is relatively small compared to the $\ell^2$ norm. Our results show that smoothness automatically enforces such concentration. This provides a principled explanation for why smooth signals are recoverable from few samples even without explicit sparsity assumptions.
\end{remark}

\begin{remark}[Extensions to other domains]
The framework developed here naturally extends to other compact Riemannian manifolds or homogeneous spaces with a well-behaved spectral decomposition, such as the torus $\mathbb{T}^d$, the rotation group $SO(3)$, or more general compact Lie groups. The key ingredients are: a Parseval identity linking spatial and spectral norms, regularity estimates that control high-frequency coefficients, and random sampling results for the associated orthogonal system. We anticipate that similar Fourier ratio bounds and recovery guarantees can be established in these settings.
\end{remark}

\begin{remark}[Connection to classical approximation theory]
Our bounds on the Fourier ratio can be seen as quantitative versions of the fact that smooth functions have rapidly decaying Fourier coefficients. The logarithmic factor in $r_N$ and $r_L$ reflects the $C^2$ smoothness. Higher regularity, for example $C^k$ with $k>2$, yields improved decay of coefficients and can be expected to lead to correspondingly smaller Fourier ratio bounds.
\end{remark}

\begin{remark}[No sampling]
The interested reader may wonder what happens if no samples are taken and one just has access to portions of the underlying field $f$. The authors of this paper developed the Fourier Ratio in the continuous setting in \cite{ILPY25} for that purpose and obtained exact recovery statements for continuous signals.
\end{remark}

\section*{Conclusion}

We have shown that smoothness of a continuous signal, quantified by membership in $C^2$, deterministically implies compressibility in the Fourier domain as measured by the Fourier ratio. This provides a rigorous foundation for stable recovery from incomplete random samples via $\ell^1$ minimization, without imposing sparsity assumptions. The results hold both for gridded Euclidean data and for bandlimited spherical signals, and the framework is extensible to other geometric settings.

The Fourier ratio serves as a natural bridge between continuous regularity and discrete compressibility, offering a unified perspective on sampling and imputation problems in imaging, remote sensing, and geometric data analysis. Future work may explore tighter bounds for higher-order smoothness, adaptive sampling strategies, and numerical validation of the recovery guarantees.

\end{document}